\documentclass[a4paper,12pt]{article}

\usepackage[T1]{fontenc} 			
\usepackage{lmodern} 				
\usepackage[english]{babel}			
\usepackage{hyperref} 				
\usepackage{paralist}				
\usepackage{graphicx}				
\usepackage{amsfonts} 				
\usepackage{amsmath} 				
\usepackage{amsopn}				
\usepackage{amsthm} 				
\usepackage{mathabx}				
\usepackage[arrow, matrix, curve]{xy} 		
\usepackage{extarrows}				
\usepackage{mathtools}				

\DeclareMathOperator{\id}{id}

\DeclareMathOperator{\im}{im}

\DeclareMathOperator{\pr}{pr}
\DeclareMathOperator{\supp}{supp}

\DeclareMathOperator{\Cent}{Cent}
\DeclareMathOperator{\ev}{ev}
\DeclareMathOperator{\End}{End}
\DeclareMathOperator{\Hom}{Hom}
\DeclareMathOperator{\Lin}{Lin}

\theoremstyle{plain}
\newtheorem{satz}{Satz}[section]		
\newtheorem{theorem}[satz]{Theorem}
\newtheorem{lemma}[satz]{Lemma}
\newtheorem{corollary}[satz]{Corollary}

\newtheorem{satz/definition}[satz]{Satz/Definition}
\newtheorem{lemma/definition}[satz]{Lemma/Definition}

\theoremstyle{definition}
\newtheorem{remark}[satz]{Remark}
\newtheorem{definition}[satz]{Definition}

\title{Universal continuous bilinear forms for compactly supported sections of Lie algebra bundles and universal continuous extensions of certain current algebras}
\author{Jan Milan Eyni}
\date{}
\allowdisplaybreaks				
\setdefaultitem{\textbullet}{-}{}{}		

\begin{document}
\maketitle

\thispagestyle{empty}

\begin{abstract}
We construct a universal continuous invariant bilinear form for the Lie algebra of compactly supported sections of a Lie algebra bundle in a topological sense. Moreover we construct a universal continuous central extension of a current algebra $A \otimes \frak{g}$ for a finite-dimensional Lie algebra $\frak{g}$ and a certain class of topological algebras $A$. In particular taking $A=C^\infty_c(M)$ for a $\sigma$-compact manifold $M$ we obtain a more detailed justification for a recent result of Janssens and Wockel concerning a universal extension for the Lie algebra $C^\infty_c(M,\frak{g})$.
\end{abstract}


\section*{Introduction and Notation}
\addcontentsline{toc}{section}{Introduction and Notations}
A continuous invariant bilinear form $\gamma$ on a Lie algebra $\frak{g}$ taking values in a locally convex space is called universal if we get  all other continuous invariant bilinear forms on $\frak{g}$ by composing $\gamma$ with a unique continuous linear map. Here invariance means that $\gamma([x,y],z) = \gamma(x,[y,z])$ for all $x,y,z \in \frak{g}$.  In \cite{Guendogan} such a bilinear form was constructed in general. To this end, one considers the quotient of the symmetric square $S^2(\frak{g})$ by the closure of the subspace generated by elements of the form $[x,y]\vee z - x\vee[y,z]$. This quotient is called $V^T_\frak{g}$ and $\kappa_\frak{g} \colon \frak{g} \times \frak{g} \rightarrow V_\frak{g}^T$, $(x,y) \mapsto [x\vee y]$ is a universal invariant bilinear form on $\frak{g}$. We recall this general approach in Section \ref{Pre}. Although any two universal invariant bilinear forms differ only by composition with an isomorphism of topological vector spaces, it is not enough to know the mere existence of a universal extension in general. Often, one would like to use more concrete realisations of universal invariant bilinear forms. This is the reason why in \cite{Guendogan}, G{\"u}ndo\u{g}an constructed a concrete universal continuous bilinear form for the Lie algebra $C^\infty(M,\frak{K})$ of sections for a given Lie algebra bundle $\frak{K}$. If $\frak{g}$ is the finite-dimensional typical fiber of $\frak{K}$ and we consider $V(\frak{K})$ to be the vector bundle with base $M$ and fibers $V(\frak{K}_p)$ for $p \in M$, G{\"u}ndo\u{g}an showed that $C^\infty(M,\frak{K}) \times C^\infty(M, \frak{K}) \rightarrow V(\frak{K})$, $(\eta , \zeta) \mapsto \kappa_\frak{g}\circ (\eta ,\zeta)$ is a universal continuous invariant symmetric bilinear form. In \cite{Bas} Janssens and Wockel constructed a universal continuous central extension of the Lie algebra $C^\infty_c(M,\frak{K})$ of compactly supported sections of a Lie algebra bundle $\frak{K}$. In their proof, they implicitly use the concept of universal continuous invariant bilinear forms for these sections. But they do not discuss whether there exists a universal continuous invariant bilinear forms for the Lie algebra $C^\infty_c(M,\frak{K})$. 
Hence the first aim of this paper is to construct a universal continuous invariant bilinear form on $C^\infty_c(M,\frak{K})$. This will be done in Section \ref{bi}. 


To show the reader the application of this special universal invariant bilinear form we recall in \ref{OmegaRes} how Janssens and Wockel used this universal invariant bilinear form in \cite{Bas}. With the help of this bilinear form they constructed a certain cocycle $\omega$  in $Z_{ct}^2(C_{c}^\infty(M,\frak{K}),E)$ for an appropriate locally convex space $E$, which they proved to be a universal cocycle. We also show the continuity of this cocycle, which was not discussed by Janssens and Wockel (see Lemma \ref{beta} and Theorem \ref{d}).



The second aim of this paper is to construct a universal continuous extension of certain so-called current algebras. In general these are algebras of the form $A \otimes \frak{g}$, where $A$ is a locally convex topological algebra and $\frak{g}$ a locally convex Lie algebra. In 2001, Maier constructed in \cite{Maier} a universal continuous central extensions for current algebras of the form $A \otimes \frak{g}$, where $A$ is a unital, commutative, associative, complete locally convex topological algebra and $\frak{g}$ a finite-dimensional semi-simple Lie algebra. The canonical example for such a current algebra is given by the smooth functions from a manifold $M$ to $\frak{g}$. To show the universality of the cocycle $\omega$ in \cite{Bas}, Janssens and Wockel used Maiers cocycle to construct a universal cocycle for the compactly supported smooth functions from a $\sigma$-compact manifold to a Lie algebra $\frak{g}$ in \cite[Theorem 7.2]{Bas}. 
G{\"u}ndo\u{g}an showed in \cite[5.1.14]{Guendogan} that the ideas from \cite{Bas} can be used to construct a universal cocycle for current algebras $A\otimes \frak{g}$ with pseudo-unital, commutative and associative algebras $A$ that are inductive limits of unital Fr{\'e}chet algebras. But this class of current algebras does not contain the compactly supported smooth maps $C^\infty_c(M,\frak{g})$ from a $\sigma$-compact finite-dimensional manifold $M$ to the Lie algebra $\frak{g}$. 
So in Section \ref{co}, we show that the cocycles constructed in \cite{Bas} respectively \cite[Theorem 5.1.14]{Guendogan} work for locally convex associative algebras $A$ which are the inductive limit of complete locally convex algebras $A_n \subseteq A$, such that we can find an element $1_n \in A$ with $1_n \cdot a =a$ for all $a \in A_n$. Obviously, this class of algebras contains the compactly supported smooth functions on a $\sigma$-compact manifold.

Throughout this paper we will use the following notations and conventions.
\begin{compactitem}
\item We write $\mathbb{N}$ for the set of integers $\{1,2,3,...\}$.
\item All locally convex spaces considered are assumed Hausdorff.
\item If $E$ is finite-dimensional vector space and $M$ a manifold, we write $C^\infty_c(M,E)$ for the space of compactly supported smooth functions from $M$ to $E$.
\item If $M$ is a manifold and $\pi \colon \mathbb{V} \rightarrow M$ a vector bundle with base $M$, we write $C^\infty(M,\mathbb{V})$ for the space of smooth sections in $\mathbb{V}$ and $C^\infty_c(M,\mathbb{V})$ for the space of compactly supported smooth sections. As usual, we write $\Omega^k(M,\mathbb{V})$ for the space of $\mathbb{V}$-valued $k$-forms on $M$ and $\Omega_c^k(M,\mathbb{V})$ for the space of compactly supported $\mathbb{V}$-valued $k$-forms on $M$.
\item Let $M$ be a $\sigma$-compact manifold, $U \subseteq M$ open, $V$ a vector space and $\mathbb{V}$ a vector bundle with base $M$. For $f \in C^\infty_c(U,V)$, $X \in C^\infty_c(U,\mathbb{V})$ and $\omega \in \Omega^k_c(U,\mathbb{V})$ respectively, we write $f_\sim$, $X_\sim$ and $\omega_\sim$ respectively, for the extension of $f$, $X$ and $\omega$ to $M$ by $0$ outside $U$.
\item If $V$ and $W$ are vector spaces, we write $\Lin(V,W)$ for the space of linear maps from $V$ to $W$ and in the case of topological vector spaces we write $\Lin_{ct} (V, W)$ for the space of continuous linear maps.
\item If $\frak{g}$ and $\frak{h}$ are Lie algebras, we write $\Hom (\frak{g}, \frak{h})$ for the space of Lie algebra homomorphisms from $\frak{g}$ to $\frak{h}$ and in the case of topological Lie algebras we write $\Hom_{ct} (\frak{g}, \frak{h})$ for the space of continuous Lie algebra homomorphisms.
\item We write $A_1$ for the unitalisation of an associative algebra $A$.
\end{compactitem}

\section{Preliminaries}\label{Pre}
In this section, we recall the basic concepts of universal continuous invariant bilinear forms. These basic definitions and results can also be found in \cite[Chapter 4]{Guendogan}. 

\begin{definition}
Let $\frak{g}$ be a Lie algebra. A pair $(V,\beta)$ with a vector space $V$ and a symmetric bilinear map $\beta \colon \frak{g} \times \frak{g} \rightarrow V$ is called an {\it invariant symmetric bilinear form} on $\frak{g}$ if $\beta ([x,y] , z) = \beta (x,[y,z])$ for all $x,y,z\in \frak{g}$. The invariant symmetric bilinear form $(V,\beta)$ is called {\it algebraic universal} if for every other invariant symmetric bilinear form $(W,\gamma)$ on $\frak{g}$, there exists a unique linear map $\psi \colon V \rightarrow W$ such that $\gamma = \psi \circ \beta$. It is clear that if $\beta$ is algebraic universal, then another invariant symmetric bilinear form $(W,\gamma)$ on $\frak{g}$ is algebraic universal if and only if we can find an isomorphism of vector spaces $\varphi \colon V \rightarrow W$ with $\gamma = \varphi \circ \beta$.
 
In the case that $\frak{g}$ is a locally convex Lie algebra and $V$ is a locally convex space, the pair $(V,\beta)$ is called {\it continuous invariant symmetric bilinear form} on $\frak{g}$ if $\beta$ is continuous and it is called {\it topological universal} or {\it universal continuous invariant symmetric bilinear form} if for every other continuous invariant symmetric bilinear form $(W,\gamma)$ on $\frak{g}$, there exists a unique continuous linear map $\psi \colon V \rightarrow W$ such that $\gamma = \psi \circ \beta$. It is clear that if $\beta$ is topological universal, then another invariant symmetric bilinear form $(W,\gamma)$ on $\frak{g}$ is topological universal if and only if we can find an isomorphism $\varphi \colon V \rightarrow W$ of topological vector spaces with $\gamma = \varphi \circ \beta$.

\cite[Remark 4.1.5. and Proposition 4.1.7.]{Guendogan} tell us that their always exists a universal continuous invariant symmetric bilinear form $(V_\frak{g}, \kappa_\frak{g})$ for a given locally convex Lie algebra $\frak{g}$.
\end{definition}

With \cite[Proposision 4.3.3]{Guendogan} we get directly the following Lemma \ref{alguni}.
\begin{lemma}\label{alguni}
For a $\sigma$-compact finite-dimensional manifold $M$ and a finite-dimensional perfect Lie algebra $\frak{g}$ the map
\begin{align*}
{\kappa_\frak{g}}_\ast \colon C^\infty_c(M,\frak{g}) \times C^\infty_c(M,\frak{g}) \rightarrow C^\infty_c(M,V_\frak{g}), (f,g) \mapsto \kappa_\frak{g}^T \circ (f,g) 
\end{align*}
is an algebraic universal symmetric invariant bilinear form. Notably the image of $\kappa_\frak{g}$ spans $C^\infty_c(M,V_\frak{g})$.
\end{lemma} 
In the case that $M$ is connected, the preceding Lemma \ref{alguni} can be found in \cite[Corollary 4.3.4]{Guendogan}.

\begin{definition}
If $\frak{g}$ is a Lie algebra, we call the subalgebra 
\begin{align*}
\left\{f \in \Lin (\frak{g}):~ (\forall x,y \in \frak{g}) ~ f([x,y]) = [f(x),y] \right\}
\end{align*}
of the associative algebra $\Lin (\frak{g})$ the {\it Centroid} of $\frak{g}$. 
\end{definition}

The following Lemma \ref{Cent1} can be found in \cite[Lemma 4.1.6]{Guendogan}.
\begin{lemma}\label{Cent1}
Let $\frak{g}$ be a Lie algebra, $W$ a vector space and $\beta \colon \frak{g} \times \frak{g} \rightarrow W$ an invariant bilinear map. Then $\beta(f(x),y) = \beta(x,f(y))$ for all $x \in [\frak{g} , \frak{g}]$, $y \in \frak{g}$ and $f \in \Cent(\frak{g})$. 
\end{lemma}

The next Lemma \ref{perfect} comes from \cite[Remark 4.2.7]{Guendogan}.
\begin{lemma}\label{perfect}
The Lie algebra $C^\infty_c(M,\frak{g})$ is perfect for every $\sigma$-compact manifold $M$ and perfect finite-dimensional Lie algebra $\frak{g}$.
\end{lemma}

\begin{definition}
Let $M$ be a manifold, $\frak{g}$ a Lie algebra and $\pi \colon \frak{K} \rightarrow M$ a vector bundle with typical fiber $\frak{g}$. If for every $m \in M$ the space $\pi^{-1}(\{m\})$ is endowed with a Lie algebra structure such that there exists an atlas of local trivialisations $\varphi \colon \pi^{-1}(U_\varphi) \rightarrow U_\varphi \times \frak{g}$ of $\frak{K}$ such that  for every $p \in U_\varphi$ the map $\varphi(p, \_) \colon \frak{K}_m \rightarrow \frak{g}$ is a Lie algebra homomorphism, then we call $\frak{K}$ a {\it Lie algebra bundle}.
\end{definition}

%
%
%


In Definition \ref{TopFuerBuendel} we endow the vector space of sections as well as compactly supported sections into a given vector bundle with a locally convex topology. We follow the definitions from \cite[Chaper 3]{GloecknerII}.
\begin{definition}\label{TopFuerBuendel} 
Let $M$ be a finite-dimensional manifold, $V$ a finite-dimensional vector space and $\pi \colon \mathbb{V} \rightarrow M$ a vector bundle with typical fiber $V$. If $\eta \in C^\infty(M,\mathbb{V})$ and $\varphi \colon \pi^{-1}(U) \rightarrow U_\varphi \times V$ is a local trivialisation of $\mathbb{V}$ we write $\eta_\varphi:= \pr_2 \circ \varphi \circ \eta|_{U_\varphi} \in C^\infty(U_\varphi,V)$ for the local representation of $\eta$. Let $\mathcal{A}$ be an atlas of $\mathbb{V}$. We give $C^\infty(M,\mathbb{V})$ the initial topology with respect to the maps $\sigma_\varphi \colon C^\infty(M,\mathbb{V}) \rightarrow C^\infty(U_\varphi , V)$, $\eta \mapsto \eta_\varphi$. \cite[Lemma 3.9]{GloecknerII} tells us that this topology does not depend on the choose of the atlas. Moreover \cite[Lemma 3.7]{GloecknerII} tells us that the topological embedding $C^\infty(M,\mathbb{V}) \rightarrow \prod_{\varphi \in \mathcal{A}} C^\infty(U_\varphi , V) \eta \mapsto (\eta_\varphi)_{\varphi \in \mathcal{A}}$ has closed image and so $C^\infty(M,\mathbb{V})$ becomes a locally convex space. Especially we see that $C^\infty(M,\mathbb{V})$ is a Fr{\'e}chet space if we find a countable atlas of local trivialisations of $\mathbb{V}$. 

If $K \subseteq M$ is compact we write $C^\infty_K(M,\mathbb{V})$ for the closed subspace of sections from $\mathbb{V}$, whose supports are contained in $K$. In the case of an countable atlas of local trivialisations of $\mathbb{V}$ it is clear that $C^\infty_K(M,\mathbb{V})$ is a Fr{\'e}chet space. We give $C_c^\infty(M,\mathbb{V})$ the topology such that it becomes the inductive limit of the spaces $C_K^\infty(M,\mathbb{V})$ in the category of locally convex spaces, where $K$ runs through all compact sets. 

If $\frak{g}$ is a finite-dimensional Lie algebra and $\frak{K}$ a Lie algebra bundle with typical fiber $\frak{g}$, we define the Lie bracket $[\_,\_] \colon C^\infty(M,\frak{K}) \times C^\infty(M,\frak{K}) \rightarrow C^\infty(M,\frak{K})$ by $[\eta,\zeta](p) = [\eta(p),\zeta(p)]$ for $\eta,\zeta \in C^\infty(M,\frak{K})$, where the latter Lie bracket, is taken in $\frak{K}_p$. Together with this Lie bracket $C^\infty(M,\frak{K})$ becomes a topological Lie algebra.
\end{definition}

\begin{lemma}
Let $M$ be a finite-dimensional $\sigma$-compact manifold, $\frak{g}$ a finite-dimensional Lie algebra and $\frak{K}$ a Lie algebra bundle with typical fiber $\frak{g}$. In this situation $C_c^\infty(M,\frak{K})$ becomes a topological Lie algebra.
\end{lemma}
\begin{proof}
The map $\frak{K} \oplus \frak{K} \rightarrow \frak{K}$ that maps $(v,w)$ to $[v,w]_{\frak{K}_p}$ for $v,w \in \frak{K}_p$ and $p \in M$ is continuous. With the $\Omega$-Lemma (see, e.g. \cite[Theorem 8.7]{Michor} or \cite[F.24]{Gloeckner}) we see that $C^\infty_c(M,\frak{K})$ is a topological Lie algebra.
\end{proof}

\section{Topological universal bilinear forms for \\ compactly supported sections of Lie\\ algebra bundles}\label{bi}
The aim of this section is to construct a universal invariant continuous bilinear form for the space of compactly supported sections into a Lie algebra bundle. To this end, we first show the ``local statement'', that means we construct a universal continuous invariant bilinear form for the compactly supported smooth functions on a $\sigma$-compact manifold into a Lie algebra $\frak{g}$ (Theorem \ref{local uni}). Afterwards we glue the local solutions together to a global one (Theorem \ref{GlobTopUn}).

\bigskip

The following fact is well known.
\begin{lemma}\label{Top}
If $X$ is a locally compact space, $(K_n)_{n \in \mathbb{N}}$ a locally finite cover by compact sets $K_n$ then there exists an open neighbourhood $V_n$  of $K_n$ in $X$ for every $n  \in \mathbb{N}$, such that $(V_n)_{n \in \mathbb{N}}$ is locally finite.
\end{lemma}

\begin{lemma}\label{QuotientenAbb}
If $M$ is a $\sigma$-compact finite-dimensional manifold, $E$ a finite-dimensional vector space
and $(\rho_m)_{m \in \mathbb{N}}$ a partition of unity
, then $\Phi \colon \bigoplus_{m \in \mathbb{N}}\allowbreak C^\infty (M,E) \rightarrow C^\infty_c(M,E)$, $(f_m)_{m \in \mathbb{N}} \mapsto \sum_{m \in \mathbb{N}} \rho_m \cdot f_m$ is a quotient map.
\end{lemma}
\begin{proof}
First we show the continuity of $\Phi$. Because $\Phi$ is linear, it is enough to show that $C^\infty(M,E) \rightarrow C^\infty_c(M,E)$, $f \mapsto \rho_m \cdot f$ is continuous for every $m \in \mathbb{N}$. The local convex space $C^\infty_c(M,E)$ is the inductive limit of spaces $C^\infty_{K_n}(M,E)$ with $n \in \mathbb{N}$, where $(K_n)_{n \in \mathbb{N}}$ is a compact exhaustion of $M$. Because the support of $\rho_m$ is compact we can find $n \in \mathbb{N}$ with $\supp(\rho_m) \subseteq K_n$. We see that the map $C^\infty(M,E) \rightarrow C^\infty_c(M,E)$, $f \mapsto \rho_m \cdot f$ takes its image in the subspace $C^\infty_{K_n}(M,E)$. Now we conclude that $\Phi$ is continuous, because $C^\infty(M,E) \rightarrow C^\infty_{K_n}(M,E)$, $f \mapsto \rho_m \cdot f$ is continuous. 

With Lemma \ref{Top} we find a locally finite cover $(V_n)_{n \in \mathbb{N}}$ of $M$, such that $V_n$ is a neighbourhood for $\supp(\rho_n)$. For $n \in \mathbb{N}$ we choose a smooth function $\sigma_n \colon M \rightarrow [0,1]$, such that $\sigma_n|_{\supp(\rho_n)}\equiv 1$ and $\supp(\sigma_n) \subseteq V_n$. Because a compact subset of $M$  is only intersected by finite many sets of the cover $(V_n)_{n\in \mathbb{N}}$ we can define the map $\Psi \colon C^\infty_c(M,E) \rightarrow \bigoplus_{n=1}^\infty C^\infty(M,E)$, $\gamma \mapsto (\sigma_n \cdot \gamma)_{n \in \mathbb{N}}$, witch is obviously a right-inverse for $\Phi$. If $K\subseteq M$ is compact, we find $N \in \mathbb{N}$, such that $K \cap V_n=\emptyset$ for $n \geq N$. We conclude $\Psi(C^\infty_K(M,E)) \subseteq \prod^N_{n=1} C^\infty(M,E) \times \{0\}$. Obviously the map $C^\infty(M,E) \rightarrow \prod_{n=1}^NC^\infty(M,E)$, $\gamma \mapsto (\sigma_n \cdot \gamma)_{n=1,..,N}$ is continuous. We conclude that $\Psi$ is a continuous linear right-inverse for $\Psi$ and so we see, that $\Phi$ is an open linear map.
%
%
%
\end{proof}

\begin{lemma}
If $M$ is a finite-dimensional $\sigma$-compact manifold and $\frak{g}$ a finite-dimensional Lie algebra, then ${\kappa_{\frak{g}}}_\ast \colon C^\infty_c(M,\frak{g})^2 \rightarrow C_c^\infty(M,V_\frak{g})$, $(f,g) \mapsto \kappa_\frak{g}\circ (f,g)$ is continuous.
\end{lemma}
\begin{proof}
The lemma follows directly from \cite[Corollary 4.17]{GloecknerIII}.
\end{proof}

\begin{theorem}\label{local uni}
Let $\frak{g}$ be a perfect finite-dimensional Lie algebra and $M$ a finite-dimensional $\sigma$-compact  manifold. Then ${\kappa_{\frak{g}}}_\ast \colon C^\infty_c(M, \frak{g})^2 \rightarrow C_c^\infty(M,V_\frak{g})$ is topological universal.
\end{theorem}
\begin{proof}
We know that ${\kappa_\frak{g}}_\ast \colon C^\infty_c(M, \frak{g})^2 \rightarrow C_c^\infty(M,V_\frak{g})$ is an algebraic universal invariant form. Moreover $\kappa^T_{C_c^\infty (M,\frak{g})} \colon C_c^\infty (M,\frak{g})^2 \rightarrow V^T_{C_c^\infty (M,\frak{g})}$ is a topological universal invariant form. 

Because ${\kappa_\frak{g}}_\ast$ is a continuous invariant bilinear map, we find a continuous linear map $f \colon V_T(C_c^\infty (M,\frak{g})) \rightarrow C_c^\infty (M,V_\frak{g})$ such that ${\kappa_\frak{g}}_\ast = f \circ \kappa^T_{C_c^\infty (M,\frak{g})}$ and because $\kappa^T_{C_c^\infty (M,\frak{g})}$ is an invariant bilinear map, we find a linear map $g\colon C_c^\infty (M,V_\frak{g}) \rightarrow V^T_{C_c^\infty (M,\frak{g})}$ with $\kappa^T_{C_c^\infty (M,\frak{g})} = g \circ {\kappa_\frak{g}}_\ast$. We get the commutative diagram
\begin{align*}
\begin{xy}
\xymatrixcolsep{5pc}\xymatrix{
C_c^\infty (M,\frak{g})^2 \ar[r]^-{{\kappa_\frak{g}}_\ast} \ar[d]_-{\kappa^T_{C_c^\infty (M,\frak{g})}} & C_c^\infty (M,V_\frak{g}) \ar@/^1pc/[dl]^-{g}\\
V^T_{C_c^\infty (M,\frak{g})} \ar@/^0.0pc/[ur]^-{f}
}
\end{xy}
\end{align*}
With $f\circ g \circ {\kappa_\frak{g}}_\ast = {\kappa_\frak{g}}_\ast$ and the fact that ${\kappa_\frak{g}}_\ast$ is algebraic universal, we get 
\begin{align}\label{dora}
f \circ g = {\id}_{C_c^\infty (M,V_\frak{g})}.   
\end{align}
Let $(\rho_m)_{m \in \mathbb{N}}$ be a partition of unity of $M$. From Lemma \ref{QuotientenAbb} we know the quotient map $\Phi$ and get the commutative diagram
\begin{align*}
\begin{xy}
\xymatrixcolsep{5pc}\xymatrix{
\bigoplus_{m \in \mathbb{N}} C^\infty (M, V_\frak{g}) \ar[r]^-{h} \ar[d]_-{\Phi} & V^T_{C^\infty_c(M,\frak{g})}\\
C_c^\infty (M,V_\frak{g}) \ar[ur]_-{g}
}
\end{xy}
\end{align*}
with $h \colon \bigoplus_{m \in \mathbb{N}} C^\infty (M,V_\frak{g}) \rightarrow V^T_{C_c^\infty(M,\frak{g})}$, $(\varphi_{m})_{m \in \mathbb{N}}  \mapsto  \sum_{m \in \mathbb{N}}  g(\varphi_{m} \cdot \rho_m)$. If we can show that $h$ is continuous, we get that also $g$ is continuous.

Because $h$ is linear it is enough to show that $C^\infty(M,V_\frak{g}) \rightarrow V^T_{C^\infty_c(M,\frak{g})}$, $\varphi \mapsto g(\varphi \cdot \rho_m)$ is continuous for all $m\in \mathbb{N}$. The space $V^T_\frak{g}$ is finite-dimensional, because $\frak{g}$ is finite-dimensional. Let $(v_i)_{i=1,...,n}$ be a basis of $V_\frak{g}$. We write $\varphi_i$ for the $i$-th component for a map $\varphi \in C^\infty(M,V_\frak{g})$. We can find ${\xi_i}_{j}, {\zeta_i}_j \in C^\infty_c(M,\frak{g})$, such that $v_i \cdot \rho_m = \sum_{j=1}^{n_i} {\kappa_{\frak{g}}}_\ast({\xi_i}_{j} , {\zeta_i}_j)$, because ${\kappa_\frak{g}}_\ast \colon C^\infty_c(M,\frak{g})^2 \rightarrow C^\infty_c(M,V_\frak{g})$ is algebraic universal and so $\im({\kappa_\frak{g}}_\ast)$ generates $C_c^\infty(M,V_\frak{g})$. For $\varphi \in C^\infty(M,V_\frak{g})$ we calculate 
\begin{align*}
&g(\varphi \cdot \rho_m) = \sum\limits_{i=1}^n g(\varphi_i \cdot \rho_m \cdot v_i) = \sum\limits_{i=1}^n \sum\limits_{j=1}^{n_i} g\left(\varphi_i \cdot {\kappa_\frak{g}}_\ast ({\xi_i}_j, {\zeta_i}_j)\right)\\
= &\sum\limits_{i=1}^n \sum\limits_{j=1}^{n_i} g\left( {\kappa_\frak{g}}_\ast (\varphi_i \cdot {\xi_i}_j, {\zeta_i}_j)\right) 
= \sum\limits_{i=1}^n \sum\limits_{j=1}^{n_i} \kappa^T_{C_c^\infty (M,\frak{g})}(\varphi_i \cdot {\xi_i}_j, {\zeta_i}_j).
\end{align*}
Because $C^\infty(M,\mathbb{R}) \rightarrow V^T_{C^\infty_c(M,\frak{g})}$, $\psi \mapsto \kappa^T_{C_c^\infty (M,\frak{g})}(\psi \cdot {\xi_i}_j, {\zeta_i}_j)$ is continuous we see that also $g$ is continuous.

Now we have $g \circ f \circ \kappa^T_{C_c^\infty (M,\frak{g})} = \kappa^T_{C_c^\infty (M,\frak{g})}$. Because $g$ is continuous we get $g\circ f = \id_{V^T_{C_c^\infty(M,\frak{g})}}$. And with (\ref{dora}), we see that $f$ is an isomorphism of topological vector spaces.  
\end{proof}


\begin{remark}
If $\frak{g}$ and $\frak{h}$ are Lie algebras and $f\colon \frak{h} \rightarrow \frak{g}$ is a Lie algebra homomorphism, then there exists a unique linear map $f_\kappa \colon V_\frak{h} \rightarrow V_\frak{g}$ with $f_{\kappa}(\kappa_\frak{h}(x,y)) = \kappa_\frak{g} (f(x),f(y))$.
\end{remark}

\begin{definition}
Let $M$ be a manifold, $\frak{g}$ a finite-dimensional Lie algebra and $\frak{K}$ a Lie algebra bundle with base $M$ and typical fiber $\frak{g}$. If $\mathcal{A}$ is an atlas of local trivialisations of $\frak{K}$, we define $V(\frak{K}):= \bigcup_{m \in M} V(\frak{K}_m)$ and the surjection $\rho \colon V(\frak{K}) \rightarrow M$, $v \mapsto m$ for $v \in \frak{K}_m$. For a local trivialisation $\varphi \colon \pi^{-1}(U_\varphi) \rightarrow U_\varphi \times \frak{g}$ we define the map $\tilde{\varphi} \colon \rho^{-1}(U_\varphi) \rightarrow U_\varphi \times V_\frak{g}$, $v \mapsto \big(\rho(v), \varphi(\rho(v),\_)_{\kappa}(v)\big)$. Together with the atlas of local trivialisations $\{\tilde{\varphi}: \varphi \in \mathcal{A} \}$ we get a vector bundle $\rho \colon V(\frak{K}) \rightarrow M$. In this article we will always write $\tilde{\varphi}$ for the trivialisation of $V(\frak{K})$ that comes from a trivialisation $\varphi$ of $\frak{K}$.
\end{definition}

\begin{definition}
For a manifold $M$, a Lie algebra $\frak{g}$ and a Lie algebra bundle $\frak{K}$ with base $M$ and typical fiber $\frak{g}$, we define the map $\kappa_\frak{K} \colon C_c^\infty(M,\frak{K})^2 \rightarrow C_c^\infty(M,V(\frak{K}))$ by $\kappa_\frak{K}(X,Y)(m)=\kappa_{\frak{K}_m}(X(m),Y(m))$ for $m \in M$.
\end{definition}

\begin{lemma}
If $M$ is a $\sigma$-compact, finite-dimensional manifold, $\frak{g}$ a finite-dimensional Lie algebra and $\frak{K}$ a Lie algebra bundle with base $M$ and typical fiber $\frak{g}$, then $\kappa_\frak{K}\colon C_c^\infty(M,\frak{K})^2 \rightarrow C_c^\infty(M,V(\frak{K}))$ is continuous.
\end{lemma}
\begin{proof}
The map $\frak{K} \oplus \frak{K} \rightarrow V(\frak{K})$ that maps $(v,w)$ to $\kappa_{\frak{K}_p}(v,w)$ for $v,w \in \frak{K}_p$ and $p \in M$ is continuous. The assertion now follows from the $\Omega$-Lemma (see, e.g. \cite[Theorem 8.7]{Michor} or \cite[F.24]{Gloeckner}).
\end{proof}

\begin{lemma}\label{Span}
The image of $\kappa_\frak{K}$ spans $C_c^\infty(M,V(\frak{K}))$, if $\frak{g}$ is a perfect finite-dimensional Lie algebra, $M$ a $\sigma$-compact finite-dimensional manifold and $\pi_\frak{K}\colon \frak{K} \rightarrow M$ a Lie algebra bundle with base $M$ and typical fiber $\frak{g}$.
\end{lemma}
\begin{proof}

To show the assertion of the lemma, we only need to show that the global statement can be reduced to the local one, because the local statement follows from Lemma \ref{alguni}. Let $\eta \in C_c^\infty(M,V(\frak{K}))$ and $K:= \supp(\eta)$. We find local trivialisations $\varphi_i \colon \pi^{-1}(U_i) \rightarrow U_i \times \frak{g}$ of $\frak{K}$ for $i = 1,..,k$ with $K\subseteq \bigcup_{i=1}^n U_i$. Let $(\lambda_i)_{i=0,..,k}$ be a partition of unity of $M$ that is subordinate to the open cover that consists of the sets $M \setminus K$ and $U_i$ for $i=1,..,k$. We get $\eta = \sum_{i=1}^k \lambda_i \cdot \eta$ and $\lambda_i \cdot \eta \in C_c^\infty(M,V(\frak{K}))$ with $\supp(\lambda_i \cdot \eta) \subseteq U_i$. The assertion now follows from the fact that $\tilde{\varphi}_i \colon \rho_{V(\frak{K})}^{-1}(U_i) \rightarrow U_i \times V_\frak{g}$ is a local trivialisation of $V(\frak{K})$.
\end{proof}

\begin{lemma}
Let $M$ be a finite-dimensional $\sigma$-compact manifold, $\frak{g}$ a finite-dimensional Lie algebra and $\frak{K}$ a Lie algebra bundle with typical fiber $\frak{g}$. If $\frak{g}$ is perfect, then also $C_c^\infty(M,\frak{K})$ is perfect.
\end{lemma}
\begin{proof}
Because the assertion holds for the local statement, it is enough to show that the global statement can be reduced to the local one. Let $\eta \in C_c^\infty(M,\frak{K})$ and $K:= \supp(\eta)$. We choose $\varphi_i$ and $\lambda_i$ analogous to the proof of Lemma \ref{Span} and calculate $\eta = \sum_{i=1}^n \lambda_i \cdot \eta$. We have $\lambda_i \cdot \eta \in C_c^\infty(M,\frak{K})$ with $\supp(\lambda_i \cdot \eta) \subseteq U_i$. This finishes the proof.
\end{proof}

In \cite[Theorem 4.4.4]{Guendogan} a local statement for algebraic universal invariant bilinear forms is used to get an analogous global statement for the space of algebraic universality of spaces of sections in a Lie algebra bundle. We now want to transfer this approach to a topological statement for compactly supported sections in a Lie algebra bundle in Theorem \ref{GlobTopUn}.
\begin{theorem}\label{GlobTopUn}
For a perfect finite-dimensional Lie-algebra $\frak{g}$, a $\sigma$-compact manifold $M$ and a Lie algebra bundle $\frak{K}$ with base $M$ and typical fiber $\frak{g}$, the map $\kappa_\frak{K} \colon C_c^\infty (M,\frak{K})^2 \rightarrow C_c^\infty (M,V(\frak{K}))$ is topological universal.
\end{theorem}
\begin{proof}
Let $(\psi_i \colon \pi^{-1}(U_{\psi_i})\rightarrow U_{\psi_i} \times \frak{g})_{i \in I}$ be a bundle atlas of $\frak{K}$ and $(\rho_i)_{i \in I}$ a partition of unity of $M$  with $\supp(\rho_i) \subseteq U_{\psi_i}$. Let $\gamma \colon C_c^\infty(M,\frak{K})^2 \rightarrow W$ be a continuous invariant bilinear form. For $i \in I$ we define 
\begin{align*}
\gamma_i \colon C_c^\infty (U_{\psi_i}, \frak{g})^2 &\rightarrow W\\
(f,g) & \mapsto \gamma \left( (\psi_i^{-1} \circ (\id , f) )_{\sim} , (\psi_i^{-1} \circ (\id , g))_{\sim} \right).
\end{align*}

The bilinear map $\gamma_i$ is an invariant symmetric bilinear form. We want to show that it is also continuous. Obviously it is enough to show that $C^\infty_c(U_{\psi_i} , \frak{g}) \rightarrow C_c^\infty(M,\frak{K})$, $f \mapsto (\psi^{-1} \circ (\id ,f))_\sim$ is continuous. This follows from the following diagram
\begin{align*}
\begin{xy}
\xymatrixcolsep{5pc}\xymatrix{
C^\infty_c(U_{\psi_i} , \frak{g}) \ar[d]_-{\id} \ar[r]^-{f\mapsto (\psi^{-1} \circ (\id ,f))_\sim}  & C_c^\infty(M,\frak{K}) \ar@{^{(}->}[d]\\
C^\infty(U_{\psi_i}, \frak{g}) \ar[r]_-{\iota_i} & \bigoplus\limits_{j\in I} C^\infty (U_{\psi_j} ,\frak{g}).
}
\end{xy}
\end{align*}

So we can find a continuous linear map $\beta_i \colon C_c^\infty (U_{\psi_i}, V_\frak{g}) \rightarrow W$, such that the diagram 
\begin{align*}
\begin{xy}
\xymatrixcolsep{5pc}\xymatrix{
C_c^\infty (U_{\psi_i} , \frak{g})^2 \ar[r]^-{\gamma_i} \ar[d]_-{{\kappa_\frak{g}}_\ast} & W\\
C_c^\infty ( U_{\psi_i} , V ( \frak{g} ) ) \ar[ur]_-{\beta_i}
}
\end{xy}
\end{align*}
commutes. 

For $i \in I$ let $\tilde{\psi_i}$ be the corresponding bundle-chart of $V(\frak{K})$ that comes from $\psi_i$. We define $\beta \colon C_c^\infty(M,V(\frak{K})) \rightarrow W$, $X \mapsto \sum_{i \in I} \beta_i \big((\rho_i \cdot X)_{\tilde{\psi_i}}\big)$. To see that $\beta$ is continuous it is enough to show that $h \colon \bigoplus_{i \in I} C^\infty(U_{\psi_i}, V_\frak{g}) \rightarrow W$, $(f_i)_{i \in I} \mapsto \sum_{i \in I} \beta_i(\rho_i \cdot f_i)$ is continuous, because $ C_c^\infty (M , V(\frak{K})) \rightarrow \bigoplus_{i\in I} \allowbreak C^\infty (U_{\psi_i} ,\allowbreak V_\frak{g}) $, $X \mapsto X_{\tilde{\psi_i}}$ is a topological embedding. To check the continuity of $h$ it is enough to show the continuity of $C^\infty(U_{\psi_i} , V_\frak{g}) \rightarrow C^\infty_c(U_{\psi_i} , V_\frak{g})$, $f \mapsto \rho_i \cdot f$. The continuity of the latter map is clear, because it takes its image in a subspace $C^\infty_K(U_{\psi_i} , V_\frak{g})$ for a compact set $K \subseteq U_{\psi_i}$.

Let $\zeta_i \colon M \rightarrow [0,1]$ be a smooth map with $\supp(\zeta_i) \subseteq U_{\psi_{i}}$ and $\zeta_i|_{\supp(\rho_i)}=1$ for $i \in I$. With Lemma \ref{Span} in mind we calculate for $X,Y \in C_c^\infty(M,\frak{K})$
\begin{align*}
&\beta(\kappa_\frak{K}(X,Y))  = \sum_{i \in I} \beta_i \left( (\rho_i \zeta_i \kappa_\frak{K}(X,Y))_{\tilde{\psi_i}} \right)  = \sum_{i\in I} \beta_i \left(\kappa_{\frak{K}}(\rho_i X, \zeta_i Y)_{\tilde{\psi_i}}\right)\\
 =& \sum_{i \in I} \beta_i {\kappa_\frak{g}}_\ast \left( (\rho_i X)_{\psi_i}, (\zeta_i Y)_{\psi_i} \right) = \sum_{i\in I} \gamma_i \left((\rho_i X)_{\psi_i}, (\rho_i Y)_{\psi_i}\right) = \sum_{i \in I} \gamma ( \rho_i X , \zeta_i Y)\\
 =& \sum_{i\in I} \gamma ( \underbracket{\zeta_i \rho_i}_{=\rho_i} X , Y) = \gamma(X,Y).
\end{align*}
Here we used that $C_c^\infty(M,\frak{K}) \rightarrow C_c^\infty(M,\frak{K})$, $X \mapsto \zeta_i \cdot X$ is in $\Cent (C_c^\infty(M,\frak{K}))$ and that $C_c^\infty(M,\frak{K})$ is a perfect Lie algebra.

The uniqueness of $\beta$ follows from Lemma \ref{Span}.
\end{proof}

\section{Continuity of a cocycle}\label{cont}

In Theorem \ref{OmegaRes} we recall how the universal continuous invariant symmetric bilinear form on the Lie algebra of compactly supported sections is used in \cite{Bas} to construct a certain cocycle $\omega$  in $Z_{ct}^2(C_{c}^\infty(M,\frak{K}),E)$ for an appropriate locally convex space $E$. But as mentioned in the introduction the construction in \cite{Bas} does not discuss the continuity of the cocycle $\omega$. We show that this cocycle is actually continuous  with the help of Lemma \ref{beta} and Theorem \ref{d}. 

\bigskip

In the following Definition \ref{TopOmega} we equip the space of $k$-forms respectively the compactly supported $k$-forms on a $\sigma$-compact finite-dimensional manifold with the usual topology, as, e.g., in \cite[Definition 5.2.6]{Guendogan}.
\begin{definition}\label{TopOmega}
Let $M$ be a finite-dimensional $\sigma$-compact manifold, $\pi \colon \mathbb{V} \rightarrow M$ a vector bundle with base $M$ and $k \in \mathbb{N}$. We give $\Omega^k(M,\mathbb{V})$ the induced topology of $C^\infty((TM)^k , \mathbb{V})$. The subspace $\Omega^k (M, \allowbreak \mathbb{V})$ is closed in $C^\infty((TM)^k, \allowbreak \mathbb{V})$, because convergence $C^\infty((TM)^k,\mathbb{V})$ implies pointwise convergence. Hence $\Omega^k(M,\mathbb{V})$ becomes a locally convex space. If the typical fiber of $\mathbb{V}$ is a Fr{\'e}chet-space, it even is a Fr{\'e}chet-space. If $K \subseteq M$ is compact, we give $\Omega^k_K (M,\mathbb{V}):= \{\omega \in \Omega^k (M,\mathbb{V}): \supp (\omega) \subseteq K\}$ the induced topology of $\Omega^k(M,\mathbb{V})$, so that $\Omega^k_K(M, \mathbb{V})= \bigcap_{p\in M\setminus K, v \in (T_pM)^k} \ev_{p,v}^{-1}(\{0\})$ with the point evaluation $\ev_{p,v} \colon \Omega^k(M,\mathbb{V}) \rightarrow \mathbb{V}_p$, $\omega \mapsto \omega_p(v)$, becomes a closed subspace of $\Omega^k(M,\mathbb{V})$ and by this a locally convex space. As usual we write $\Omega_c^k(M,\mathbb{V})$ for the space of compactly supported $k$-forms and equip it with the topology, such that it becomes the inductive limit of the spaces $\Omega^k_K(M, \mathbb{V})$.
\end{definition}

We now want to 
fix our notation concerning
$k$-forms and connections and recall some basic facts. All this is well known, for instance see \cite{Darling} and \cite[Chapter 2.2. and 2.3.]{Guendogan}.
\begin{definition}
Let $M$ be a finite-dimensional $\sigma$-compact manifold, $\mathbb{V}$ a vector bundle with base $M$, $\frak{K}$ a Lie-algebra-bundle with base $M$ and $k \in \mathbb{N}$.
\begin{compactenum}[(a)]
\item The space $\Omega^k_c(M,\mathbb{V})$ becomes a $C^\infty (M,\mathbb{R})$-module by the multiplication $C^\infty(M,\mathbb{R}) \times \Omega^k_c(M,\mathbb{V}) \rightarrow \Omega^k_c(M,\mathbb{V})$, $(f,\omega)  \mapsto f\cdot \omega$ with $(f\cdot \omega)_p = f(p) \cdot \omega_p$.
\item We get a bilinear map $\Omega^k_c(M,\mathbb{R}) \times C^\infty(M,\mathbb{V}) \rightarrow \Omega_c^k(M,\mathbb{V})$ $(\omega , \eta)  \mapsto \omega \cdot \eta$ with $ (\omega \cdot \eta)_p(v_1,..,v_k) = \omega_p(v_1,..,v_k) \cdot \eta(p)$ for $v_i \in T_pM$. 
\item We call a $\mathbb{R}$-linear map $d \colon C_c^\infty(M,\mathbb{V}) \rightarrow \Omega_c^1 (M,\mathbb{V})$ {\it Koszul connection}, if $d(f\eta) = fd\eta + \eta df$ for all $\eta\in C^\infty(M,\mathbb{V})$ and $f \in C^\infty (M,\mathbb{R})$.



\item We define the continuous $C^\infty(M,\mathbb{R})$-linear map $C_c^\infty(M,\frak{K}) \times \Omega_c^1(M,\frak{K}) \rightarrow \Omega^1_c(M,\frak{K})$, $(\eta,\omega) \mapsto [\eta,\omega]$ with $([\eta,\omega])_p(v) = [\eta(p),\omega_p(v)]$. Moreover we set $[\omega,\eta] = -[\eta,\omega]$.
\item We call a Koszul connection $D\colon C_c^\infty(M,\frak{K}) \rightarrow \Omega^1_c (M,\frak{K})$ a {\it Lie connection}, if $D[\eta,\tau] = [D\eta,\tau] + [\eta,D\tau]$ for  $\eta,\tau \in C_c^\infty(M,\frak{K})$.
\end{compactenum}
\end{definition}

The following fact in Lemma \ref{LieConnection} should be part of the folklore. For instance see \cite[Remark 2.3.14.]{Guendogan}.
\begin{lemma}\label{LieConnection}
For every finite-dimensional $\sigma$-compact manifold $M$ and Lie-algebra-bundle $\pi \colon \frak{K} \rightarrow M$, there exists a continuous Lie connection \\$D \colon C_c^\infty(M,\frak{K}) \rightarrow \Omega^1_c(M,\frak{K})$.
\end{lemma}

\begin{lemma}\label{beta}
If $M$ is a finite-dimensional $\sigma$-compact manifold and $\pi \colon \frak{K} \rightarrow M$ a Lie-algebra-bundle with finite-dimensional typical fiber, then we define the map 
\begin{align*}
\tilde{\kappa}_{\frak{K}} \colon \Omega_c^1(M,\frak{K}) \times C_c^\infty(M,\frak{K}) \rightarrow \Omega_c^1(M, V(\frak{K}))
\end{align*}
by $(\tilde{\kappa}_{\frak{K}}(\omega,\eta))_p(v) =  \kappa_{\frak{K}_p}(\omega_p(v) ,\eta(p))$. The map $\tilde{\kappa}_\frak{K}$ is $C^\infty(M,\mathbb{R})$-bilinear and continuous. If moreover $D\colon C_c^\infty(M,\frak{K}) \rightarrow \Omega^1_c(M,\frak{K})$ is a continuous Lie connection, then 
\begin{align*}
\beta \colon C_c^\infty(M,\frak{K})^2 \rightarrow \Omega_c^1(M,V(\frak{K})),~ (\zeta,\eta) \mapsto \tilde{\kappa}_{\frak{K}}(D\zeta,\eta) + \tilde{\kappa}_{\frak{K}} (\zeta , D\eta)  
\end{align*}
is a continuous, invariant, symmetric bilinear form. 
\end{lemma}
\begin{proof}
To show the continuity of $\beta$ we only have to prove that $\tilde{K}_{\frak{K}}$ is continuous. The map 
\begin{align*}
(T^\ast M \otimes \frak{K}) \oplus \frak{K} \rightarrow \Lin(TM, V(\frak{K})), (\lambda \otimes v ,w) \mapsto \kappa_{\frak{K}_p}(\lambda (\_) \cdot v,w)
\end{align*}
is continuous. With the identifications $\Omega^1_c(M,\frak{K})\cong C^\infty_c(M,T^\ast M \otimes \frak{K})$ and $\Omega^1_c(M,V(\frak{K}))\cong C^\infty_c(M, \Lin(TM, V(\frak{K})))$ the continuity follows from the $\Omega$-Lemma (see, e.g. \cite[Theorem 8.7]{Michor} or \cite[F.24]{Gloeckner}).

We show that $\beta$ is invariant: 
\begin{align*}
&\beta([\eta_1,\eta_2],\eta_3)_p(v) = \kappa ((D[\eta_1,\eta_2])_p(v),\eta_3(p)) + \kappa( [\eta_1(p) , \eta_2(p) ] , (D \eta_3)_p(v))\\
=& \kappa ([(D\eta_1)_p(v) , \eta_2(p)] , \eta_3(p))  + \kappa ([\eta_1 (p) , (D \eta_2)_p(v)], \eta_3(p))\\
&+ \kappa( [\eta_1(p) , \eta_2(p) ] , (D \eta_3)_p(v) ) \\
= &\kappa ( (D\eta_1)_p(v) , [\eta_2(p) ,\eta_3(p) ]) + \kappa ( \eta_1 (p) , [ (D\eta_2)_p(v) , \eta_3(p)])\\
&+ \kappa (\eta_1 (p) , [ \eta_2(p) , (D\eta_3)_p(v) ])\\
=& \kappa ( (D\eta_1)_p(v) , [\eta_2(p) ,\eta_3(p) ]) + \kappa (\eta_1(p) , D[\eta_2, \eta_3]_p(v)) = \beta (\eta_1 ,[\eta_2,\eta_3])_p(v)
\end{align*}

The rest of the statement is clear.
\end{proof}

As in \cite{Bas} we construct a Koszul connection for the vector bundle $C_c^\infty(M,V \allowbreak (\frak{K}))$. This is also the point where we apply Theorem \ref{GlobTopUn}, to see that the Koszul connection is actually continuous. Moreover we need the continuity of the map $\beta$ from Lemma \ref{beta}. 
\begin{theorem}\label{d}
Let $M$ be a finite-dimensional $\sigma$-compact manifold, $\pi \colon \frak{K} \rightarrow M$ a Lie-algebra-bundle with perfect, finite-dimensional typical fiber $\frak{g}$, $D\colon C_c^\infty \allowbreak (M,\frak{K}) \rightarrow \Omega^1_c(M,\frak{K})$ a continuous Lie connection and $\beta$ as in Lemma \ref{beta} a continuous, invariant, symmetric bilinear form. Then there exists a unique continuous Koszul connection $d\colon C_c^\infty(M,V(\frak{K})) \rightarrow \Omega_c^1(M , \frak{K})$ such that the diagram 
\begin{align*}
\begin{xy}
\xymatrix{
C_c^\infty(M,\frak{K})^2 \ar[r]^-{\beta} \ar[d]_-{\kappa_{\frak{K}}} & \Omega_c^1 (M,V(\frak{K}))\\
C_c^\infty(M,V(\frak{K}))\ar[ur]_-{d}
}
\end{xy}
\end{align*}
commutes.
\end{theorem}
\begin{proof}
With Theorem \ref{GlobTopUn} we find a unique continuous $\mathbb{R}$-linear map $d\colon C_c^\infty(M, \allowbreak V(\frak{K})) \rightarrow \Omega_c^1(M , \frak{K})$ such that the above diagram commutes. Now we have to show that $d(f \cdot \eta) = df \cdot \eta + f \cdot d\eta$ holds for all $f\in C^\infty(M,\mathbb{R})$ and $\eta \in C_c^\infty(M,V(\frak{K}))$. Because the image of $\kappa_{\frak{K}}$ spans $C^\infty(M,V(\frak{K}))$, it is sufficient to show the assertion for $\eta = \kappa_{\frak{K}}(\xi,\zeta)$ with $\xi,\zeta\in  C_c^\infty(M,V(\frak{K}))$.
\begin{align*}
&d(f\cdot \kappa_{\frak{K}}(\xi,\zeta))_p(v) = (d(\kappa_{\frak{K}}(f\xi,\zeta)))_p(v)\\
=& \kappa (D(f \cdot \xi)_p(v) , \zeta(p)) + \kappa (f(p)\cdot \xi(p), (D \zeta)_p(v))\\
=& \kappa (df(v) \cdot \xi(p) , \zeta(p)) + \kappa ( f(p) (D\xi)_p(v) , \zeta(p) ) + \kappa (f(p) \xi(p) , D\zeta_p(v))\\
=& df(v) \cdot \kappa(\xi(p) , \zeta(p) ) + f(p) \kappa (D\xi_p(v) , \zeta(p)) + f(p) \kappa(\xi(p) , D\zeta_p(v))\\
=& (df \cdot \kappa (\xi ,\zeta))_p(v) + f(p) \cdot (\beta(\xi,\zeta))_p(v) \\
=& (df \cdot \kappa (\xi ,\zeta))_p(v) + f(p) \cdot (d\kappa(\xi,\zeta))_p(v).
\end{align*}
\end{proof}

In Theorem \ref{OmegaRes} we now repeat the construction of the cocycle from \cite{Bas} and find with Theorem \ref{d} that this cocycle is actually continuous.
\begin{theorem}\label{OmegaRes}
Let $M$ be a finite-dimensional $\sigma$-compact manifold, $\pi \colon \frak{K} \rightarrow M$ a Lie-algebra-bundle with perfect, finite-dimensional typical fiber $\frak{g}$ and $D\colon C_c^\infty(M,\frak{K}) \rightarrow \Omega^1_c(M,\frak{K})$ a continuous Lie connection. We use the Koszul connection $d \colon C_c^\infty(M,V(\frak{K})) \rightarrow \Omega^1_c(M,V(\frak{K}))$ constructed in Theorem \ref{d}, write $\overline{\Omega}_c^1(M,V(\frak{K})) := \Omega_c^1(M,V(\frak{K}))/(dC_c^\infty(M,V(\frak{K})))$ and define the map
\begin{align*}
\omega \colon C_c^\infty(M,\frak{K})^2 \rightarrow \overline{\Omega}_c^1(M,V(\frak{K})),~ (\eta,\zeta) \mapsto [\tilde{\kappa}_{\frak{K}}(D\eta ,  \zeta)]
\end{align*}
with the help of $\tilde{\kappa}_{\frak{K}}$ of Lemma \ref{d}. Then $\omega$ is a cocycle and also continuous, i.e., $\omega \in H^2_{ct}(C_c^\infty(M \allowbreak , \frak{K}), \allowbreak \overline{\Omega}_c^1(M \allowbreak ,V(\frak{K})))$.
\end{theorem}
\begin{proof}
The continuity of $\omega$ follows from Lemma  \ref{d}. The $\mathbb{R}$-bilinearity is clear. That $\omega$ is alternating follows from $\tilde{\kappa}_{\frak{K}}(D\eta ,  \xi) + \tilde{\kappa}_{\frak{K}}(D \xi , \eta) \in dC^\infty(M,V(\frak{K}))$. 
Moreover, we have:
\begin{align*}
&\sum\limits_{\sigma \in \mathcal{A}_3} \omega([\eta_{\sigma(1)} , \eta_{\sigma(2)}] , \eta_{\sigma(3)})  = \kappa ([\eta_1 , \eta_2 ], D\eta_3) + \kappa([\eta_2 , \eta_3] , D\eta_1) \\
&+ \kappa ([\eta_3 , \eta_1] , D\eta_2).
\end{align*}
On the other hand we have:
\begin{align*}
&\kappa ([\eta_1 , \eta_2] , D\eta_3) = -\kappa(\eta_3 , D[\eta_1 ,\eta_2]) = -\kappa(\eta_3 , [D\eta_1 ,\eta_2]) - \kappa(\eta_3 , [\eta_1 ,D\eta_2])\\
=& -\kappa([D\eta_1 , \eta_2] , \eta_3) - \kappa ([\eta_3 ,\eta_1] , D\eta_2) = -\kappa([\eta_2 , \eta_3] , D\eta_1) - \kappa ([\eta_3 ,\eta_1] , D\eta_2).
\end{align*}
This finishes the proof.
\end{proof}

\section{Universal continuous extensions of certain current algebras}\label{co}
Maier constructed in \cite{Maier} a universal cocycle for current algebras with a unital complete locally convex algebra. In \cite[Theorem II.7]{Bas} Janssens and  Wockel showed that this cocycle works also for the algebra of compactly supported functions on a $\sigma$-compact finite-dimensional manifold. 
G{\"u}ndo\u{g}an showed in \cite{Guendogan} that this approach also works for a certain class of locally convex pseudo-unital algebras. But the class of locally convex algebras he considers does not contain the compactly supported functions on a $\sigma$-compact finite-dimensional manifold. Our aim in this section is to use the ideas from \cite{Guendogan} to show that the cocycle constructed in \cite{Bas} respectively \cite{Maier} actually works, in a topological sense for a class of locally convex algebras {\it without unity} that contains the compactly supported smooth functions on a $\sigma$-compact algebra.

\bigskip

In Definition \ref{Aog} we first recall the basic concept of current algebras that should be part of the folklore. 
\begin{definition}\label{Aog}
If $A$ is a commutative  pseudo unital algebra and $\frak{g}$ a finite-dimensional Lie algebra, we endow the tensor product $A \otimes \frak{g}$ with the unique Lie bracket that satisfies $[a\otimes x ,b\otimes y] = ab \otimes [x,y]$ for $a,b \in A$ and $x,y \in \frak{g}$. We endow $A \otimes \frak{g}$ with the topology of the projective tensor product of locally convex spaces. This Lie algebra is even a locally convex algebra as one can see in \cite[Remark 2.1.9.]{Guendogan}.
Moreover \cite[Remark 4.2.7]{Guendogan} tells us that $A\otimes \frak{g}$ is perfect if $\frak{g}$ is so.
\end{definition}



\begin{definition}
Let $\frak{g}$ be a locally convex Lie algebra and $V$ a complete locally convex space that is a trivial $\frak{g}$-module and $\omega\in Z^2_{ct}(\frak{g},V)$ a cocycle.  We call $\omega$  {\it weakly universal} for complete locally convex spaces if for every complete locally convex space $W$ considered as trivial $\frak{g}$-module, the map $\delta_W\colon \Lin_{ct}(V,W)\rightarrow H_{ct}^2(\frak{g},W)$, $\theta \mapsto [\theta\circ \omega]$ is bijective.
\end{definition}

The following Lemma \ref{schwach}, can be found in \cite[Lemma 1.12 (iii)]{Neeb}.
\begin{lemma}\label{schwach}
Let $\frak{g}$ be a locally convex Lie algebra and $V$ a complete locally convex space that is a trivial $\frak{g}$-module and $\omega\in Z^2_{c}(\frak{g},V)$ a weak universal cocycle. If $\frak{g}$ is topological perfect, then $\omega$ is universal for all complete locally convex spaces in the classic sense. 
\end{lemma}

\begin{remark}
Actually the lemma in \cite{Neeb} requires the condition that the considered extension is weak universal for the underlying field $\mathbb{K}$ of the vector spaces. But this condition is only used in the proof of statement 1.12 (ii). The proof of statement (iii) neither requires statement 1.12 (ii) nor this condition.
\end{remark}

We remind the reader of the concept of topological universal differential module. This concept is for example developed in \cite[Chapter 5.2]{Guendogan} or \cite{Maier}.
\begin{definition}
For a complete locally convex associative commutative unital algebra $A$, a pair $(E,D)$ with a complete locally convex $A$-module $E$ and continuous derivation $D \colon A \rightarrow E$ of $E$ is called {\it universal topological differential module} of $A$ if we find a unique continuous linear map $\varphi \colon E \rightarrow F$ such that
\begin{align*}
\begin{xy}
\xymatrix{
A \ar[r]^-{T} \ar[d]_-{D}  & F\\
E \ar[ru]_-{\varphi}  &
}
\end{xy}
\end{align*}
commutes for every complete locally convex $A$-module $F$ and continuous $F$-derivation $T \colon A \rightarrow F$. \cite[Chapter 5.2]{Guendogan} or \cite{Maier} tell us that there always exists a universal topological differential module $(\Omega(A),d_A)$ for a given complete locally convex associative commutative unital algebra $A$.
\end{definition}


Now we recall the cocycle for current algebras of the form $A \otimes \frak{g}$ with unital complete locally convex algebra $A$ that Maier proved in \cite{Maier} to be universal.
\begin{definition}
If $\frak{g}$ is a finite-dimensional semisimple Lie algebra, $A$ a commutative, associative unital complete locally convex algebra, then we define the cocycle 
\begin{align*}
\omega_{\frak{g},A} \colon A\otimes \frak{g} \times A \otimes \frak{g} &\rightarrow V_\frak{g} \otimes (\Omega(A)/\overline{d_A(A)})\\
(a \otimes x , b \otimes y) &\mapsto \kappa_\frak{g}(x,y) \otimes [a \cdot d_A(b)].
\end{align*} 
For convenience we write $V_{\frak{g},A}:=V_\frak{g} \otimes (\Omega(A)/\overline{d_A(A)})$.
\end{definition}

In \cite[Theorem 16]{Maier} we directly get the following Lemma \ref{Maieruniversell}.
\begin{lemma}\label{Maieruniversell} 
Let $\frak{g}$ be a finite-dimensional semisimple Lie algebra, $A$ a commutative, associative unital complete algebra. If $W$ is a complete locally convex space considered as trivial $\frak{g}$-module, then the map $\delta_W\colon Lin(V_{\frak{g},A},W)\rightarrow H^2(A\otimes\frak{g},W)$, $\theta \mapsto [\theta\circ \omega_{\frak{g},A}]$ is bijective.
\end{lemma}

\begin{definition}\label{delta}
For a locally convex algebra $A$ and a finite-dimensional Lie algebra $\frak{g}$ the Lie algebra $A\otimes \frak{g}$ is locally convex. For $y\in \frak{g}$, we get a continuous bilinear map $A\times \frak{g} \rightarrow A\otimes \frak{g}$ $(c,x) \mapsto c\otimes [y,x]$ which induces a continuous linear map $\delta_y \colon A\otimes \frak{g} \rightarrow A\otimes \frak{g}$ with $\delta_y (c\otimes x)= c \otimes [y,x]$. We get a linear map $\delta \colon \frak{g} \rightarrow \End_V (A \otimes \frak{g})$, $y\mapsto \delta_y$. Moreover $\delta$ is a Lie algebra homomorphism, because $\delta_{[y_1,y_2]}(c \otimes x) = c \otimes [[y_1,y_2],x] = c \otimes [y_1,[y_2,x]] - c \otimes [y_2,[y_1,x]] = \delta_{y_1} \delta_{y_2} (c\otimes x) - \delta_{y_2}\delta_{y_1} (c\otimes x)$. Also we have $\delta_y \in der(A\otimes \frak{g})$, because $\delta_y( [ c\otimes x , c'\otimes x'] ) = cc' \otimes [y,[x,x']] = cc'\otimes [[y,x],x'] + cc' \otimes [x,[y,x']] = [\delta_y (c\otimes x), c'\otimes x'] + [c \otimes x , \delta_y (c' \otimes x')]]$. We define $[y,\_]:=\delta_y$ for $y\in \frak{g}$.

With the Lie algebra homomorphism $\delta$ we can define the semidirect product $(A\otimes \frak{g}) \rtimes \frak{g}$ with the Lie-bracket $[(z_1,y_1),(z_2,y_2)] = ([z_1,z_2] + \delta_{y_1}(z_2) - \delta_{y_2}(z_1) , [y_1,y_2]) = ([z_1,z_2] + [y_1,z_2] - [y_2,z_1],[y_1,y_2])$ for $z_i\in A\otimes \frak{g}$ and $y_i \in \frak{g}$, where we wrote $[y,\_]:=\delta_y$ for $y\in \frak{g}$. $A\otimes \frak{g} \rtimes \frak{g}$ becomes a locally convex Lie algebra. 

We identify $A \otimes \frak{g}$ with the ideal $\im (i)$, where $i \colon A \otimes \frak{g} \rightarrow A \otimes \frak{g} \rtimes \frak{g}$, $z \mapsto (z,0)$ is a topological embedding that is a Lie algebra homomorphism. The image is closed as kernel of the projection $A \otimes \frak{g} \rtimes \frak{g} \rightarrow \frak{g}$. 

Moreover we identify $\frak{g}$ with the subalgebra $\im (i_\frak{g})$, where $i_\frak{g} \colon \frak{g} \rightarrow (A\otimes \frak{g}) \rtimes \frak{g}$, $x\mapsto (0,x)$ is a topological embedding with closed image that is a Lie algebra homomorphism.

For $(z,x) \in A \otimes \frak{g} \rtimes \frak{g}$ we can write $(z,x)=(z,0)+(0,x)=z+x$. 
\end{definition}


\cite[Remark 5.1.7 and Lemma 5.1.8]{Guendogan} lead to the following Lemma \ref{IsoUnital}.
\begin{lemma}\label{IsoUnital}
For a locally convex algebra $A$ and a finite-dimensional Lie algebra $\frak{g}$ we have an isomorphism of locally convex Lie algebras $\varphi \colon A_1 \otimes \frak{g} \rightarrow (A \otimes \frak{g}) \rtimes \frak{g}$ with $(\lambda,a) \otimes w \mapsto (a \otimes w,\lambda w)$ for all $\lambda \in \mathbb{K}$, $a \in A$ and $w \in \frak{g}$. 
\end{lemma}

\begin{lemma}
If $\frak{g}$ and $\frak{h}$ are locally convex Lie algebras, $V$ a locally convex space and $\varphi \colon \frak{g} \rightarrow \frak{h}$ a continuous Lie algebra homomorphism, than $H_{ct}(\varphi) \colon H^2_{ct}(\frak{h},V) \rightarrow H^2_{ct}(\frak{g},V)$, $[\omega] \mapsto [\omega \circ (\varphi, \varphi)]$ is a well-defined and linear map. 
\end{lemma}
\begin{proof}
For $\eta \in \Lin(\frak{h},V)$ and $\omega \in Z^2_{ct}(\frak{h},V)$ with $\omega = \eta \circ [\_,\_]$ we have $(\varphi,\varphi)^\ast(\omega) = \eta \circ \varphi \circ [\_,\_]$.
\end{proof}

We will use the concept of neutral triple evolved in \cite[Definition 5.1.3]{Guendogan} and recall it in the next Definition \ref{spez.Trip.}.
\begin{definition}\label{spez.Trip.}
Let be $A$ a pseudo-unital associative commutative locally convex algebra $A$ and $\frak{g}$ a finite-dimensional perfect Lie algebra.
\begin{compactenum}[(a)]
\item We get an $A$-module structure $\cdot\colon A \times (A \otimes \frak{g}) \rightarrow A\otimes \frak{g}$ with $a \cdot (b \otimes y) (a\cdot b) \otimes y$ for $a,b \in A$ and $y \in \frak{g}$. Actually $A\otimes \frak{g}$ is a $A$-module in the category of locally convex spaces, because $\frak{g}$ is finite-dimensional. In this situation we call $\nu \in A$ neutral for $f \in A\otimes \frak{g}$, if $\nu \cdot f =f$.
\item For $f \in A \otimes \frak{g}$ (resp. $\varphi \in A$) we call $(\lambda,\nu,\mu) \in A^3$ a {\it neutral triple} for $f$ (resp. $\varphi$), if $\mu \cdot f =f$ (resp. $\mu \cdot \varphi = \varphi$), $\nu \cdot \mu = \mu$ and $\lambda \cdot \nu = \nu$.
\item Let $(v_i)_{i=1,..,n}$ be a basis of $\frak{g}$. For $f= \sum_{i=1}^n \varphi_i \otimes v_i \in A \otimes \frak{g}$ with $\varphi \in A$ we choose $\mu_f \in A$ with $\mu$ neutral for all $\varphi_i$. Moreover we choose $\nu_f$ neutral for $\mu_f$ and $\lambda_f$ neutral for $\nu_f$. It is clear that $(\lambda_f,\nu_f,\mu_f)$ is a neutral triple for $f$ and for all $\varphi_i$. For the rest of this section we will use this notation to talk about this element. For every $\varphi\in A$ we choose a neutral triple $(\lambda_\varphi,\nu_\varphi, \mu_\varphi)$. Obviously $(\lambda_\varphi,\nu_\varphi, \mu_\varphi)$ is a neutral triple for  $\varphi \otimes v$ for every $v \in\frak{g}$. 
\end{compactenum}
\end{definition}


%
%

From the proof of \cite[Theorem 5.1.10]{Guendogan} we can extract the following Lemma \ref{wohl}.
\begin{lemma}\label{wohl}
Let $A$ be a pseudo-unital associative commutative locally convex algebra, $\frak{g}$ a finite-dimensional perfect Lie algebra and $V$ a locally convex space, $\omega \in Z^2_{ct}(A \otimes \frak{g} , V)$, $f \in A \otimes \frak{g}$ and $y \in \frak{g}$. If $(\lambda_1,\nu_1, \mu_1)$ and $(\lambda_2,\nu_2, \mu_2)$ are neutral triples for $f$, then $\omega(f,\lambda_1 \otimes y) = \omega(f, \lambda_2 \otimes y)$.
\end{lemma}

We now prove that \cite[Theorem 5.1.14]{Guendogan} also holds for a class of locally convex algebras that contains the compactly supported smooth functions on a $\sigma$-compact finite-dimensional manifold.
\begin{theorem}\label{calg}
Let $A$ be a complete pseudo-unital algebra that is the strict inductive limit of locally convex subalgebras $A_m \subseteq A$ for which we can find an element $1_m \in A$ with $1_m \cdot a=a$ for all $a \in A_m$. Moreover, let $\frak{g}$ be a semisimple finite-dimensional Lie algebra, $V$ a locally convex space and $i \colon A\otimes\frak{g} \rightarrow A\otimes\frak{g} \rtimes \frak{g}$ the natural inclusion. Then $H^2_{ct}(i) \colon H^2_{ct}(A\otimes\frak{g}\rtimes \frak{g},V) \rightarrow H^2_{ct}(A\otimes\frak{g},V)$ is bijective.
\end{theorem}
\begin{proof}
Surjectivity: Let $(v_i)_{i=1,..,n}$ be a basis of $\frak{g}$. We use the notation of Definition \ref{spez.Trip.}. Given $\omega_0 \in Z_{ct}^2(A\otimes\frak{g},V)$, we define  $\omega \colon A\otimes\frak{g} \rtimes \frak{g} \rightarrow V$ $(f_1,y_1), (f_2,y_2) \mapsto \omega_0(f_1,f_2) + \omega_0(f_1 , \lambda_{f_1} \otimes y_1) - \omega_0(f_2, \lambda_{f_2} \otimes y_1)$. 

For $f,g \in A\otimes\frak{g}$, $r \in \mathbb{K}$ and $y \in \frak{g}$ we can choose a neutral triple $(\lambda, \nu , \mu)$ that is neutral for $f$ and for $g$. Especially this triple is also neutral for $rf+g$. Because we now get $\omega_0 (rf+g, \lambda_{rf+g} \otimes y) = \omega_0(rf+g,\lambda \otimes y) = r\omega_0(f,\lambda \otimes y) + \omega_0(g, \lambda \otimes y) = r\omega_0(f,\lambda_f \otimes y) + \omega_0(g, \lambda_g \otimes y)$, we can easily proof that $\omega$ is bilinear. Obviously $\omega$ is anti-symmetric.   

The argument that $\omega$ is a cocycle works exactly like in the proof \cite[5.10.]{Guendogan}. For the convenience of the reeder we will recall this argument in the appendix.

We now want to show that $\omega$ is also continuous. We just have to show that the bilinear map $\psi \colon A\otimes\frak{g} \times \frak{g} \rightarrow V$, $(f,y) \mapsto \omega_0(f, \lambda_f \otimes y)$ is continuous. Because we can identify $(A\otimes\frak{g}) \otimes \frak{g}$ with $(A \otimes \frak{g})^n$, it is sufficient to prove the continuity of $(A\otimes\frak{g})^n \rightarrow V$, $(f_i)_{i=1,..,n} \mapsto \sum_{i=1}^n \omega_0(f_i, \lambda_{f_i}\otimes v_i)$. To show the continuity of $A\otimes \frak{g} \rightarrow V$, $f \mapsto \omega_0(f, \lambda_f \otimes v)$, we again identify $A\otimes \frak{g}$ with $A^n$ and prove the continuity of $A^n \rightarrow V$, $(\varphi_i)_{i=1,..,n} \mapsto \sum_{i=1}^n \omega_0(\varphi_i \otimes v_i, \lambda_{f} \otimes y$ with $f = \sum_{i=1}^n \varphi_i \otimes v_i$ and an arbitrary $y \in \frak{g}$. But because of the construction of the neutral triple $(\lambda_f, \nu_f , \mu_f)$ we get $\omega_0(\varphi_i \otimes v_i, \lambda_{f} \otimes y) = \omega_0(\varphi_i \otimes v_i, \lambda_{\varphi_i} \otimes y)$ for $i \in \{1,...,n\}$. It remains to show that the linear map $A \rightarrow V$, $\varphi \mapsto \omega_0(\varphi \otimes x, \lambda_\varphi \otimes y)$ is continuous for $x,y \in \frak{g}$. But for $m \in \mathbb{N}$  we find an element $1_m\in A$ with $1_m \cdot a =a$ for all $a \in A_m$. We choose an element $\tilde{1}_m\in A$ that is unital for $1_m$ and an element $\tilde{\tilde{1}}_m$ that is unital for $\tilde{1}_m$ and see that $(\tilde{\tilde{1}}_m, \tilde{1}_m, 1_m)$ is a unital triple for every $\varphi \in A_m$. We see that $A_m \rightarrow V$, $\varphi \mapsto \omega_0(\varphi \otimes x, \lambda_\varphi \otimes y) = \omega_0(\varphi \otimes x, \tilde{\tilde{1}}_m \otimes y)$ is continuous and conclude that also the map $A \rightarrow V$, $\varphi \mapsto \omega_0(\varphi \otimes x \lambda_\varphi \otimes y)$ is continuous, because $A$ is the inductive limit of the subalgebras $A_m \subseteq A$.

The equation $H_{ct}^2(i)([\omega]) = [\omega_0]$ is easily checked, because for $f,g \in A\otimes\frak{g}$ we have $\omega \circ (i,i) (f,g) = \omega(f,g) = \omega_0(f,g)$.

The argument that $H_{ct}^2(i)$ is injective works exactly like in the proof \cite[5.10.]{Guendogan}. For the convinience of the reeder we will recall this argument in the appendix.
\end{proof}

\begin{remark}\label{VerG}
It is clear that Theorem \cite[Theorem 5.1.14]{Guendogan} is a direct consequence of Theorem \ref{calg}. Moreover we see that the class of locally convex algebras $A$ considered in \ref{calg} contains the compactly supported functions on a finite-dimensional $\sigma$-compact manifold $M$.
\end{remark}

\begin{remark}\label{VerB}
Although the class of algebras in \ref{calg} contains the compactly supported smooth functions on a $\sigma$-compact manifold $M$ considered in \cite[Theorem 2.7.]{Bas}, it contains other interesting algebras, e.g. the compactly supported continuous functions on a $\sigma$-compact manifold. So Theorem \ref{calg} can also be understood as a generalisation of \cite[Theorem 2.7.]{Bas}. Also in \cite[Theorem 2.7.]{Bas} Janssens and Wockel do not discuss if the constructed cocycle $\omega$ that is mapped to $\omega_0$ by $H_{ct}^2(i)$ is actually continuous. For this point an argument is given in the above proof.
\end{remark}

The application of Theorem \ref{calg} is not very difficult.
\begin{corollary}\label{amEnde}
Let $A$ be a locally convex commutative and associative algebra, such that it is the inductive limit  of complete subalgebras $A_n \subseteq A$ with $n \in \mathbb{N}$, such that we find for every $n \in \mathbb{N}$ an element $1_n \in A$ with $1_n \cdot a = a$ for all $a \in A_n$. Moreover let $\frak{g}$ be a finite-dimensional perfect Lie algebra, then $\omega_{\frak{g},A} \colon A\otimes\frak{g} \times A\otimes\frak{g} \rightarrow V_{\frak{g},A_1}$ with $(a \otimes x, b\otimes y) \mapsto \kappa_\frak{g}(x,y) \otimes [a \cdot d_{A_1}(b)]$  is a universal cocycle for $A\otimes \frak{g}$.
\end{corollary}
\begin{proof}
The assertion follows directly from Lemma \ref{schwach} and Theorem \ref{calg}.
\end{proof}

The transition of the fact stated in Corollary \ref{amEnde} for current algebras of the form $A\otimes \frak{g}$ to the special case of Lie algebras of the form $C^\infty_c(M,\frak{g})$ can be done like in \cite[Chapter 5.2.]{Guendogan} or \cite[Theorem II.7.]{Bas}. We recall this appraoch in Remark \ref{Anwendung}.

\begin{remark}\label{Anwendung}
\cite[Theorem 11]{Maier} tells us that if $M$ is a $\sigma$-compact manifold,  $\Omega^1(C^\infty_c(M)_1)$ the universal $C^\infty_c(M)_1$-module in the category of complete locally convex spaces, then $d \circ \pr_1 \colon C^\infty_c(M)_1 \rightarrow \Omega^1_c(M)$ induces an isomorphism of topological $C^\infty_c(M)_1$-modules $\Omega^1(C^\infty_c(M)_1) \rightarrow \Omega_c^1(M)$.

Let be $\frak{g}$ a semi simple Lie algebra. For $f \in C^\infty_c(M,\frak{g})$ and $\eta \in \Omega^1_c(M,\frak{g})$ we define the $1$-form $\kappa_\frak{g}(f,dg) \in \Omega^1_c(M,V_\frak{g})$, by $\kappa_\frak{g}(f,\eta)_p(v):= \kappa_\frak{g}(f(p), \eta_p(v))$. Because $d(C^\infty_c(M,\mathbb{R}^n))$ is closed in $\Omega^1_c(M,\mathbb{R}^n)$ the map
\begin{align*}
C^\infty_c(M, \frak{g}) \times C^\infty_c(M, \frak{g}) &\rightarrow \Omega_c(M,V_\frak{g})/d(C^\infty_c(M,V_\frak{g}))\\
(f,g) & \mapsto [\kappa_\frak{g}(f,dg)].
\end{align*}
is a universal cocycle for all complete locally convex spaces.
\end{remark}

%
%
%

\section*{Appendix}

{\bf Details for the proof of \ref{calg}:} The following argument follows the proof of \cite[Theorem 5.1.10]{Guendogan}. We use the notation from the proof of Theorem \ref{calg}.

In order to show that $\omega \in Z^2(A\otimes\frak{g}\rtimes \frak{g},V)$ we choose $f,g,h \in A\otimes\frak{g}$ and $x,y,z \in \frak{g}$. First we mention the trivialities $d\omega(f,g,h) = d\omega_0(f,g,h)=0$ and $d\omega(x,y,z)=0$. We can choose a trivial $(\lambda , \nu, \mu)$ that is neutral for $f$ and $g$, and we can write $f= \sum_{i =1}^n f_i \otimes v_i$ as well as $g=\sum_{j=1}^n g_j \otimes v_j$. We calculate
\begin{align*}
\lambda \cdot [f,g] = \sum\limits_{i,j} \lambda f_i g_j \otimes [v_i,v_j] = [\lambda f,g] = [f,g]
\end{align*}
and see that $(\lambda, \nu, \mu)$ is a neutral triple for $[f,g]$. Now we calculate
\begin{align*}
  & d\omega(f,g,y) = \omega([f,g],y) + \omega([g,y],f)  + \omega([y,f],g)  \\
= & \omega_0([f,g],\lambda \otimes y) + \omega_0([g,y], f) + \omega_0([y,f],g)\\
= & \omega_0([f,g],\lambda \otimes y) + \omega_0([g, \lambda \otimes y], f) + \omega_0([\lambda \otimes y, f] , g)
= & d\omega_0 (f,g,\lambda \otimes y) =0.
\end{align*}
To check that $\omega$ is a cocycle we calculate 
\begin{align*}
&d\omega((f,x), (g,y) (h,z)) = d\omega(f,g,h) + d\omega(f,g,z) + d\omega(f,y,h) + d\omega(f,y,z)\\
+& d\omega(x,g,h) + d\omega(x,g,z) + d\omega(x,y,h) + d\omega(x,y,z) = 0.
\end{align*}

It remains to show the injectivity of the map $H_{ct}^2(i)$. Let $\omega \in Z_{ct}^2(A\otimes\frak{g} \rtimes \frak{g},V)$ with $\omega\circ (i,i)=\eta \circ [\_,\_]$ for $\eta \in \Lin_{ct}(A\otimes\frak{g},V)$. We define the continuous linear map $\eta' \colon A\otimes\frak{g} \rtimes \frak{g} \rightarrow V$, $(f,v)\mapsto \eta(f)$. We define the cocycle $\omega':=\omega - \eta' \circ [\_,\_]$ on $A\otimes\frak{g}\rtimes \frak{g}$. If we can show $[\omega']=0$ in $H_{ct}^2(A\otimes\frak{g},V)$ we are done. First of all we have $\omega'(f,g) =\omega(f,g) - \eta' \circ [f,g]= \omega \circ (i,i) (f,g) - \eta ([f,g])=0$ for all $f,g \in A\otimes\frak{g}$. For $f,g \in A\otimes\frak{g}$ and $y \in \frak{g}$ we calculate 
\begin{align*}
0 = -d\omega'(f,g,y) = \omega'([f,g],y) + \underbracket{\omega'([g,y],f)}_{=0} + \underbracket{\omega'([y,f],g)}_{=0} = \omega'([f,g],y).
\end{align*}
Because $A\otimes\frak{g}$ is perfect, we get that $\omega'$ equals $0$ on $A\otimes\frak{g} \times \frak{g}$ in terms of the natural identifications. For $f_1,f_2 \in A\otimes\frak{g}$ and $y_1,y_2 \in \frak{g}$ we have
\begin{align*}
\omega'((f_1,y_1) , (f_2,y_2)) = \underbracket{\omega'(f_1,f_2)}_{=0} + \underbracket{\omega'(y_1,f_2)}_{=0} + \underbracket{\omega'(f_1,y_2)}_{=0} + \omega'(y_1,y_2) .
\end{align*}
Because $\frak{g}$ is a subalgebra of $A\otimes\frak{g} \rtimes \frak{g}$ we get $\omega|_{\frak{g}\times \frak{g}} \in Z^2_{ct}(\frak{g},V)$ and because $\frak{g}$ is semisimple, we get with the Whitehead theorem for locally convex spaces, stated in \cite[Corollary A.2.9]{Guendogan}, that $H_{ct}^2(\frak{g},V) = \{0\}$. Therefore, we find $\eta'' \in \Lin_{ct}(\frak{g},V)$ with $\omega|_{\frak{g} \times \frak{g}} = \eta'' \circ [\_,\_]$. Finally we see $\omega' = \eta''' \circ [\_,\_]$ with $\eta''' \colon A\otimes\frak{g} \rtimes \frak{g} \rightarrow V$, $(f,v) \mapsto \eta''(v)$.


\begin{thebibliography}{Levy2001}
\addcontentsline{toc}{section}{References}

\bibitem{Bourbaki}
Bourbaki, N., 
``Topological Vector Spaces,'' 
Springer-Verlag, 1987.

\bibitem{Darling}
Darling, R.W.R.,
``Differential Forms and Connections,''
Cambridge University Press, 1994.


\bibitem{GloecknerIII}
Gl{\"o}ckner, H.,
{\it Lie group structures on quotient groups and universal complexifications for infinite-dimensional Lie groups,}
J. Funct. Anal. 194 (2002), 347-409.


\bibitem{Gloeckner}
Gl{\"o}ckner, H.,
\emph{Lie groups over non-discrete topological fields,}
arXiv:math/0408008v1 August 1, 2004.

\bibitem{GloecknerII}
Gl\"ockner, H.,
\emph{Differentiable mappings between spaces of sections,}
arXiv:1308.1172v1, 2013.

\bibitem{Gloeckner u. Neeb}
Gl\"{o}ckner, H.; Neeb, K.-H.,
``Infinite-Dimensional Lie Groups,''
book in preparation.

\bibitem{Guendogan}
G{\"u}ndo\u{g}an, H.,
``Classification and Structure Theory of Lie Algebras of Smooth Sections,''
Logos Berlin, 2011.

\bibitem{Bas}
Janssens, B.; Wockel, C.,
\emph{Universal central extensions of gauge algebras and groups,}
J. Reine Angew. Math. 682 (2013), 129-139.

\bibitem{Jarchow}
Jarchow, H.,
``Locally Convex Spaces,''
Teubner, 1981.

\bibitem{Maier}
Maier, P.,
\emph{Central extensions of topological current algebras,}
Banach Center Publ. 55 (2002), 61-76.

\bibitem{Michor}
Michor, P., ``Manifolds of Differentiable Mappings,''
Shiva Publishing, 1980.

\bibitem{Neeb}
Neeb, K.-H.,
\emph{Universal central extensions of Lie groups,}
Acta Appl. Math. 73 (2002), 175-219.



\end{thebibliography}
\end{document}